\documentclass{amsart} % Required for inserting images
\usepackage[dvipsnames]{xcolor}
\usepackage{amssymb,amsmath,amscd,graphicx, enumitem,hyperref,
latexsym,amsthm,tikz-cd}
\usepackage{cleveref}
\usepackage{url}
\usepackage[margin=1.5in]{geometry}
\title{Low degree subvarieties of universal hypersurfaces}

\newcommand{\Aff}{\mathbb{A}}
\newcommand{\C}{\mathbb{C}}

\newcommand{\PP}{\mathbb{P}}

\newcommand{\Z}{\mathbb{Z}}

\newcommand{\Kbar}{{\overline{K}}}

\newcommand{\calL}{\mathcal{L}}

\newcommand{\calO}{\mathcal{O}}

\newcommand{\calX}{\mathcal{X}}
\newcommand{\calY}{\mathcal{Y}}

\DeclareMathOperator{\codim}{codim}

\DeclareMathOperator{\Gal}{Gal}
\DeclareMathOperator{\Gr}{Gr}
\DeclareMathOperator{\Hom}{Hom}

\DeclareMathOperator{\Res}{Res}

\DeclareMathOperator{\Spec}{Spec}

\DeclareMathOperator{\Sym}{Sym}

\renewcommand{\leq}{\leqslant}
\renewcommand{\geq}{\geqslant}
\usepackage{mathtools}
\DeclarePairedDelimiter{\floor}{\lfloor}{\rfloor}
\DeclarePairedDelimiter{\ceil}{\lceil}{\rceil}

\theoremstyle{plain}
\newtheorem{theorem}{Theorem}[section]
\newtheorem{proposition}[theorem]{Proposition}
\newtheorem{corollary}[theorem]{Corollary}
\newtheorem{lemma}[theorem]{Lemma}
\theoremstyle{definition}
\newtheorem{example}[theorem]{Example}

\newtheorem{notation}[theorem]{Notation}
\newtheorem{question}[theorem]{Question}
\newtheorem{conjecture}[theorem]{Conjecture}
\theoremstyle{remark}
\newtheorem{remark}[theorem]{Remark}

\author{Yifeng Huang}
\address{Yifeng Huang, Department of Mathematics, University of Southern California\\Los Angeles, CA 90089\\United States}
\email{yifenghu@usc.edu}
\urladdr{\url{https://yifeng-huang-math.github.io}}

\author{Borys Kadets}
\address{Borys Kadets, Einstein Institute of Mathematics\\ Hebrew University of Jerusalem\\
Jerusalem, Israel}
\email{kadets.math@gmail.com}
\urladdr{\url{http://bkadets.github.io}}

\author{Olivier Martin}
\address{Olivier Martin, Instituto de Matemática Pura e Aplicada, Rio de Janeiro, Brazil }
\email{olivier.martin@impa.br}
\urladdr{\url{https://impa.br/~olivier.martin/}}

\begin{document}

\begin{abstract}
We study irreducible subvarieties of the universal hypersurface $\mathcal{X}/B$ of degree $d$ and dimension $n$. We prove that when $d$ is sufficiently large, a degree $kd$ subvariety $Z$ which dominates $B$ comes from intersection with a family of degree $k$ projective varieties parametrized by $B$. This answers a question raised independently by Farb and Ma. Our main tools consist of a Grassmannian technique due to Riedl and Yang, a theorem of Mumford--Ro\u \i tman on rational equivalence of zero-cycles, and an analysis of Cayley--Bacharach conditions in the presence of a Galois action.  We also show that the large degree assumption is necessary; for $d=3$, rational points are dense in $\text{Sym}^dX_{k(B)}$, and in particular are not collinear. 
\end{abstract}
\maketitle

\section{Introduction}

Consider the \emph{universal} complex degree $d$ hypersurface of dimension $n$ as a variety $X_{n,d}$ over the function field $K$ of $\mathbb{P}H^0(\mathbb{P}^{n+1}, \calO(d))$; concretely, this is the hypersurface defined by the degree $d$ homogeneous polynomial $\sum {a_{I}}x_0^{i_0}\cdots x_n^{i_n}$ over $K=\C(a_{I})$, where every coefficient $a_I$ is treated as an independent variable. We are interested in describing subvarieties $Z \subset X_{n,d}$ defined over $K$.\\

A standard cohomological or cycle-theoretic calculation shows that the index of $X_{n,d}$ equals $d$; in other words, the degree of every cycle on $X_{n,d}$ is divisible by $d$. It is also clear that $X_{n,d}$ contains cycles of every degree divisible by $d$, since intersecting $X_{n,d}$ with a general degree $k$ subvariety produces a degree $kd$ cycle on $X_{n,d}$. It is therefore natural to ask whether such simple constructions account for all effective cycles of degree $kd$ of $X_{n,d}$.

\begin{question}\label{mainquestion}
Under which conditions on $n,d,$ and $k$ do all degree $kd$ effective cycles on $X_{n,d}$ arise as transverse intersections with a degree $k$ subvariety?
\end{question}

\Cref{mainquestion} has been considered in the literature in the case of zero-cycles, or, equivalently, closed points on $X_{n,d}$. The problem of describing degree $kd$ points on $X_{n,d}$ was raised by Farb in \cite[Problem 3.5]{Farb} in the spirit of the Franchetta conjecture. In \cite[Question 2]{ma2022monodromy}, \cite{ma2024rational}, Ma raised the problem of collinearity of degree $d$ points in $X_{n, d}$, which is effectively answered by \Cref{thm:main-introduction} and \Cref{negative1}.\\

Our main result shows that for $d \gg n,k,$ every degree $kd$ cycle does indeed arise as a transverse intersection. More precisely, we prove the following.

\begin{theorem}\label{thm:main-introduction}
Suppose $d,n,k$ are integers satisfying 
 \[ d\geq 4 k n + 3 k^3 - k^2 + 1. \]
 Then every degree $kd$ effective cycle on $X_{n,d}$ is contained in a degree $k$ subvariety none of whose components lie on $X_{n,d}$.
\end{theorem}

We do not know to what extent the assumption $d \gg n,k$ is necessary. It seems possible that for some universal hypersurfaces in the Fano range $d<n+2$ all degree $d$ points are contained in lines. However, some lower bound on the degree is necessary: for hypersurfaces of degree $3$, most points of degree $3$ do not arise from line sections. 

\begin{theorem}\label{negative1}
For $n\geq 2$, degree $3$ points on $X_{n,3}$ are dense when viewed as rational points on $\Sym^3 X_{n,3}$. In particular, they are not all supported on lines.
\end{theorem}

The proof of \Cref{negative1} relies on the unirationality of cubic hypersurfaces, but does not generalize easily to other low-degree hypersurfaces of high dimension. The issue lies in understanding the degrees of the field extensions of $K$ over which a unirational parametrization can be defined.\\

The Bombieri--Lang conjecture, and its geometric versions \cite{xie2023partial}, implies that the situation of \Cref{negative1} is atypical: for large enough $d$, the variety $\Sym^d X_{n,d}$ is of general type, and so should not have a dense set of rational points. From this point of view, \Cref{thm:main-introduction} affirms the Bombieri--Lang conjecture, and moreover gives an explicit description of the exceptional locus.

\subsection{Sketch of the main argument}

The proof of \Cref{thm:main-introduction} is split into two parts: a proof in the case of zero-cycles, and a reduction of the general case to that of zero-cycles. We first explain the method used in the first step, and return to the reduction procedure in the end.\\

The analogue of \Cref{thm:main-introduction} for complete intersection curves follows easily and without any restrictions on the degree from the projective normality of complete intersection curves. We present the proof here, as it motivates our approach.
\begin{lemma}\label{curves}
A degree $k(d_1\cdots d_m)$ rational point of the universal complete intersection curve $X/K$ of multidegree $(d_1, \dots, d_m)$ arises as a transverse intersection of $X$ with a degree $k$ hypersurface defined over $K$.
\end{lemma}
\begin{proof}
A standard calculation, which we reproduce in Lemma \ref{lemma:chow-group-calculation}, shows that $\mathrm{Pic} \,X\cong \Z$ and is generated by $\mathcal{O}_X(1)$, which lies in degree $d_1\cdots d_m$. Thus any degree $k(d_1\cdots d_m)$ point $P$ of $X$ lies in the linear system $|\mathcal{O}_{X}(k)|$. By projective normality\footnote{The projective normality of a complete intersection in $\mathbb{P}^{m+1}$ follows from the properties of Koszul complexes on regular sequences.} of $X$, the following map is surjective \[H^0(\mathbb{P}_{K}^{m+1},\mathcal{O}_{\mathbb{P}^{m+1}}(k))\longrightarrow H^0(X,\mathcal{O}_{X}(k)).\] It follows that $P$ is the complete intersection of $X$ with a degree $k$ hypersurface in $\mathbb{P}^{m+1}_K$. Since $P$ is a closed point, this intersection is necessarily transverse.
\end{proof}

One can apply essentially the same strategy in higher dimension provided the following analogue of projective normality for zero-cycles holds.
\begin{quote}($\ast$) If a zero-cycle $\alpha$ on a general hypersurface $X\subset \mathbb{P}^{n+1}$ of dimension $n$ is rationally equivalent to the intersection of $X$ with a curve of degree $k$, then $\alpha$ is indeed the intersection of $X$ with a curve of degree $k$.\end{quote} However, this condition does not hold in general, as can be shown by considering the case of three points on a cubic surface; see \Cref{cor:cubic-on-cubic}.\\

Yet, a weak version of this property does hold under some degree assumptions. To state this result, we recall the notion of the Cayley--Bacharach condition $\mathit{CB}(r)$. A collection of points in a projective space is said to satisfy $\mathit{CB}(r)$ if any degree $r$ hypersurface passing through all but one of the points passes through all of them; see Section \ref{sec:Mumford+CB} for details. We use the following proposition in place of property ($\ast$).

\begin{proposition}\label{projnormCB}
Let $X\subset\mathbb{P}^{n+1}_\C$ be a very general hypersurface of degree $d\geq 2n+3$. If an effective zero-cycle $p_1+\cdots+p_{kd}$ of degree $kd$ on $X$ is rationally equivalent to $k\cdot c_1(\mathcal{O}_X(1))^{n}$ then a subset of $\{p_1,\dots,p_{kd}\}$ satisfies $\mathit{CB}(d-2n-2)$.
\end{proposition}

For $k=1$, we can further simplify the statement by appealing to  \cite[Lemma 2.4]{bcd2014gonality} (see \Cref{lem:bcd}) to geometrically interpret the Cayley--Bacharach condition. This gives the following more explicit statement.

\begin{corollary}\label{projnorm}
Let $X\subset\mathbb{P}^{n+1}_\C$ be a very general hypersurface of degree $d\geq 4n+3$. If an effective zero-cycle $p_1+\cdots+p_d$ of degree $d$ on $X$ is rationally equivalent to $c_1(\mathcal{O}_X(1))^{n}$, then there is a line in $\mathbb{P}^{n+1}_\C$ containing at least $d-2n$ of the points $p_1,\dots,p_d$ (counted with multiplicity).
\end{corollary}

\Cref{projnorm} implies \Cref{thm:main-introduction} for closed points of degree $d$: taking $p_1, \dots, p_d$ to be the Galois conjugates of a degree $d$ point, the transitivity of the Galois action ensures that any line that contains more than half of the points $p_1, \dots, p_d$ contains all of them. In the case $k>1$, it is more difficult to pass from the Cayley--Bacharach condition on subsets to the desired existence of a degree $k$ curve covering \emph{all} points, and the problem is connected to some currently open problems \cite[Question 1.1]{picoco2023geometry}. We are nevertheless able to sidestep this issue by directly analyzing, in Section \ref{sec:CB-bootstrap}, the Cayley--Bacharach condition on a subset in the presence of a transitive group action. \\

To discuss the proof of \Cref{projnormCB}, we consider $\pi: \mathcal{X}_{n,d}\longrightarrow B_{n,d}$, the universal family of hypersurfaces of degree $d$ and dimension $n$, where
\[\mathcal{X}_{n, d} \coloneqq \left\{(x,f) : f(x)=0\right\}\subset \mathbb{P}_\mathbb{C}^{n+1}\times B_{n,d},\]
and $B_{n,d} = \mathbb{P}H^0(\mathbb{P}_\mathbb{C}^{n+1},\mathcal{O}(d)).$
The variety $X_{n,d}$ is the generic fiber of $\pi$, and the degree $kd$ points on $X_{n,d}$ can be interpreted as rational multisections of $\pi$. \Cref{projnormCB} is proved by using a well-known theorem of Mumford and Ro\u \i tman (\Cref{thm:mumford}) on rational equivalence of zero-cycles. This theorem implies the following. Suppose  $X$ is a smooth hypersurface of degree $d$ and $Y\subset \text{Sym}^dX$ is an irreducible subvariety that parametrizes effective zero-cycles $\alpha_y$ which are rationally equivalent to one another and which ``span $X$'': a generic $x$ in $X$ lies in the support of $\alpha_y$ for some $y\in Y$. Then for generic $y\in Y$, a subset of the support of $\alpha_y$ satisfies $\mathit{CB}(d-n-2)$.\\

We seek to apply this statement to zero-cycles rationally equivalent to $$k\cdot c_1(\mathcal{O}_{\mathcal{X}_{n,d,b}}(1))^{n},$$ for $b\in B_{n,d}$ generic, that arise from a rational multisection of $\pi$. However, to apply this result we need a family of zero-cycles rationally equivalent to one another which ``spans" $\mathcal{X}_{n,d,b}$. This is where the work of Riedl and Yang from \cite{RY22} enters the picture. It allows to show that if for generic $b\in B_{n,d}$ there is a degree $kd$ effective zero-cycle on $\mathcal{X}_{n,d,b}$ whose support does not contain a subset satisfying $\mathit{CB}(d-2n-2)$, then, for $m>n$ and $b'\in B_{m,d}$ generic, there is a family of such cycles on $\mathcal{X}_{m,d,b'}$. This family gets larger with $m$ and eventually must span $\mathcal{X}_{m,d,b'}$, thereby contradicting the Mumford--Ro\u \i tman theorem.\\

Riedl and Yang study loci $Z_{n,d}\subset \mathcal{X}_{n,d}(\C)$ in the universal hypersurface of degree $d$ and dimension $n$. They suppose that these loci satisfy a natural compatibility condition, namely, if $(x,X)\in Z_{n,d}$ can be realized as a linear slice of $(x',X')\in \mathcal{X}_{n',d}(\C)$, then $(x',X')\in Z_{n',d}$. Of course, many loci of interest satisfy this compatibility condition. For instance $Z_{n,d}$ could consist of pairs $(x,X)$ such that $x$ lies on a rational curve in $X$.\\

What Riedl and Yang prove is that if $Z_{n,d}\subset \mathcal{X}_{n,d}(\C)$ has positive codimension, then $\codim Z_{n-r,d}\geq r+1$ for $r\geq 0$, thereby allowing to show in some cases that $Z_{n-r,d}$ is empty when $(n-r)$ is small. In short, this growth of codimension is due to the fact that a given hypersurface of degree $d$ and dimension $(n-r)$ can be realized in many ways as a linear slice of a hypersurface of degree $d$ and dimension $n$. \Cref{projnormCB} is shown by choosing an appropriate locus $Z_{n,d}$. The key fact that $Z_{n,d}\subset \mathcal{X}_{n,d}(\C)$ has positive codimension is proven using the theorem of Mumford and Ro\u \i tman.\\

Finally, to obtain the result for higher-dimensional cycles, we reduce to the zero-dimensional case. To solve the question for curves of degree $3d$, for example, we need to prove that if a general hyperplane section of a projective curve $Z$ lies on a degree $3$ curve, then $Z$ itself lies on a cubic surface. Questions of this type appear in the literature under the names ``lifting problem'' and ``generalized Laudal's lemma''. We appeal to a result of this type due to Chiantini and Ciliberto \cite{chi-cil}. Some care is needed to account for the fact that curves produced by \Cref{thm:main-introduction} (for zero-cycles) may not be geometrically irreducible. We deal with this technicality in \Cref{sec:cycles}, which completes the proof of \Cref{thm:main-introduction}. \\ 

 The use we make of the Cayley--Bacharach conditions and their relation to the Mumford--Ro\u \i tman theorem is heavily inspired by the recent literature on measures of irrationality of hypersurfaces (see \cite{bcd2014gonality,BDPELU}). In that setting, sufficient positivity of the canonical bundle ensures that a hypersurface of large degree does not admit a generically finite rational map of low degree to projective space. Indeed, fibers of such a map would satisfy a Cayley--Bacharach condition with respect to the canonical bundle, thereby ensuring that they lie on a line.

\subsection*{Acknowledgments}
This project grew out of the workshop \emph{Degree $d$ points on algebraic surfaces} hosted by the \emph{American Institute of Mathematics}. We would like to warmly thank AIM, its staff, as well as the organizers of the event. We would also like to thank Shamil Asgarli for suggesting this problem as well as all the participants who discussed this problem with us. We thank Qixiao Ma for suggesting that an earlier version of our main result might generalize to higher-dimensional cycles. Finally, we are indebted to Jake Levinson for pointing us to some references on Cayley--Bacharach conditions. The third-named author is funded by the Serrapilheira Institute, the FAPERJ, and the CNPq.

\section{Setup and notation}

For the rest of the paper, we will use the same notation as in the introduction: \begin{itemize}
\item $B_{n,d}$: the parameter space $\PP H^0(\PP^{n+1}, \calO(d))$ for degree $d$ hypersurfaces,
\item $\pi: \calX_{n,d} \longrightarrow B_{n,d}$: the universal family of hypersurfaces,
\item $\mathcal{X}_{n,d,b}$: the fiber of $\pi$ over $b\in B_{n,d}$,
\item $X_{n,d} \to \Spec \mathbb{C}(B_{n,d})$: the generic fiber of $\pi$.
\end{itemize}
We will drop the indices $n$ and $d$ if they are clear from the context.\\

In this paper, a \emph{(sub)variety} is a (sub)scheme of finite type over a field that is reduced but not necessarily closed or irreducible.\\

A \emph{rational multisection of degree $e$} of the family $\pi$ is any of the following equivalent notions:
\begin{itemize}
\item a $K=\mathbb{C}(B)$-point of $\text{Sym}^e_B X$,
\item a rational section of $\text{Sym}^e_B\mathcal{X}\longrightarrow B$,
\item an effective cycle $\mathcal{Y}=\sum_{i=1}^l a_iY_i\in Z^n(\mathcal{X})$, such that $Y_i$ is prime, $\deg(Y_i/B)>0$ for all $i$, and $\sum_{i=1}^l a_i\deg(Y_i/B)=e$.\\
\end{itemize}

Finally, a well-known Chow group calculation shows that the index of $X_{n,d}$ is equal to $d$. We give a proof here for completeness. The key point is that the total space $\calX_{n,d}$ is a projective bundle over a projective space, making it possible to calculate $\textup{CH}^\bullet(\calX)$.

\begin{lemma}\label{lemma:chow-group-calculation}
Suppose $\calY$ is a rational multisection of $\calX_{n,d}/B_{n,d}$ of degree $e$ viewed as a cycle on $\calX_{n,d}$. For a general $b \in B$, let $\calY_b$ denote the fiber of $\calY$ over $b$, viewed as an effective zero-cycle of degree $e$ in $\calX_{n,d,b}$. Then $d|e$ and 
\[\mathcal{Y}_b=\frac{e}{d} c_1(\calO_{\calX_{n,d,b}}(1))^{n}\in \textup{CH}_0(\mathcal{X}_{n,d,b}).\]
\end{lemma}
\begin{proof}
The projection $\pi':\mathcal{X}\longrightarrow \mathbb{P}^{n+1}$ of $\mathcal{X}$ onto the first factor  is a projective bundle and thus the Chow group of $\mathcal{X}$ is generated as a ring by ${\pi'}^* c_1(\mathcal{O}_{\mathbb{P}^{n+1}}(1))$ and $\pi^*\textup{Pic}(B)$. Clearly cycles in the ideal $\pi^*\text{Pic}(B)\cdot \text{CH}^\bullet(\mathcal{X})$ pullback to zero under the inclusion $\iota_b: \mathcal{X}_b\longrightarrow \mathcal{X}$. Thus, given a subvariety $\mathcal{Y}\subset \mathcal{X}$ which is of relative dimension zero over $B$, the pullback $\calY_b=\iota_b^*(\calY)$ is a multiple of $$\qquad\qquad \iota_b^*\left(\pi'^*c_1(\mathcal{O}_{\mathbb{P}^{n+1}}(1))^{n}\right)=\iota_b^*\left(\pi'^*c_1(\mathcal{O}_{\mathbb{P}^{n+1}}(1))\right)^{n}=c_1(\calO_{\calX_{b}}(1))^{n}.\qquad\qquad \qedhere$$  
\end{proof}

\section{The Mumford--Ro\u \i tman  theorem and Cayley--Bacharach conditions}\label{sec:Mumford+CB}

In this section we present the relation between the Mumford--Ro\u \i tman  theorem on rational equivalence of zero-cycles and Cayley--Bacharach conditions. To formulate the Mumford--Ro\u \i tman  theorem, let us recall that given smooth projective varieties $X$ and $Y$ and a morphism $\phi: Y\longrightarrow \textup{Sym}^d X$, Mumford defines a map
\begin{align*}\text{Tr}(\phi): H^0(X,\Omega^j)&\longrightarrow H^0(Y,\Omega^j)\\
\omega \ \ \ \ &\longmapsto  \ \phi^*(\omega^{(d)}),\end{align*}
which to a holomorphic form $\omega\in H^0(X,\Omega^j)$ associates the pullback by $\phi$ of the form $\sum_{i=1}^d\text{pr}_i^*(\omega)\in H^0(X^d,\Omega^j)$ viewed as a holomorphic form $\omega^{(d)}$ on $\textup{Sym}^d X$. Of course, $\textup{Sym}^d X$ being a singular variety, some care must be taken to define this pullback, but it behaves as expected on the smooth locus.
\begin{theorem}[{Mumford--Ro\u \i tman \cite{mumford1968rational,roitman1980rational}}]
Let $X$ be a smooth projective variety and $Y$ be a smooth variety. Consider a morphism $\phi: Y\longrightarrow \textup{Sym}^d X$ such that the class $\phi(y)\in \textup{CH}_0(X)$ is independent of $y\in Y$. Then the map 
$\textup{Tr}(\phi): H^0(X,\Omega^j)\longrightarrow H^0(Y,\Omega^j)$
vanishes for all $j>0$.
\label{thm:mumford}
\end{theorem}

Recall that a set $\{x_1,\dots,x_m\}$ of distinct points in a smooth projective variety $X$ is said to satisfy the \emph{Cayley--Bacharach condition} with respect to an effective line bundle $\calL$ if any section of $\calL$ vanishing on $(m-1)$ of the $x_i$'s vanishes on all of them. The Mumford--Ro\u \i tman  theorem implies that given a family of rationally equivalent effective zero-cycles on $X$ whose supports cover $X$, the generic member has a subset of its support satisfying the Cayley--Bacharach condition with respect to $K_X$.

\begin{corollary}\label{corCB}
Let $X$ be a smooth projective variety, $Y$ a variety, and consider a morphism $\phi: Y\longrightarrow \textup{Sym}^d X$ such that:
\begin{enumerate}
\item the class $\phi(y)\in \textup{CH}_0(X)$ is independent of $y\in Y$,
\item for $x\in X$ generic, there is a $y\in Y$ such that $x$ belongs to the support of $\phi(y)$.
\end{enumerate}
Then for $y\in Y$ generic, there is a subset of the support of $\phi(y)$ (ignoring multiplicity) satisfying the Cayley--Bacharach condition for $K_X$.
\label{cor:rat-equiv-implies-cb}
\end{corollary}
\begin{proof}
We may assume without loss of generality that $Y$ is irreducible. Consider the fiber product
\[
\begin{tikzcd}
\tilde{Y} \ar[d,"\varpi"] \ar[r,"\tilde{\phi}"] & X^d\ar[d,"\pi"]\\
Y \ar[r,"\phi"]& \text{Sym}^d X,
\end{tikzcd}
\]
where $\pi$ is the quotient map. Let $\mathrm{pr}_i:X^d\longrightarrow X$ denote the $i$-th projection map, and write $\eta_i\coloneqq \mathrm{pr}_i\circ \tilde\phi:\tilde{Y}\to X$. Hence, for $\tilde y\in \tilde Y$ we may write $\tilde \phi(\tilde y)=(\eta_1(\tilde y),\dots,\eta_d(\tilde y))$.\\

Since $\{y\in Y: x\in \textup{Supp}(\phi(y))\}\subset Y$ is non-empty for $x\in X$ generic, by taking a component of $\tilde Y$ and passing to an open subset, there is a smooth irreducible variety $\tilde Y'\subset \tilde Y$ such that $\tilde\phi(\tilde Y')\subset X^d$ dominates some factor. Note that $\varpi|_{\tilde Y '}: \tilde Y' \longrightarrow Y$ is dominant, so the image of a general point $y \in \tilde Y '$ is general in $Y$. Let $\emptyset\neq I\subset\{1,\cdots, d\}$ consist of the integers $i$ such that $\eta_i:\tilde{Y}'\longrightarrow X$ is dominant. Then for $y\in \tilde{Y}'$ generic and $x_i=\eta_i(y)$, the induced map $$\eta_{i*}:T_{\tilde Y',y}\longrightarrow T_{X,x_i}$$
is surjective if and only if $i\in I$. As a result, the induced map
$$\eta_i^*:\Omega^n_{X,x_i}\longrightarrow \Omega^n_{\tilde{Y}',y}$$ is injective if $i\in I$ and zero if $i\notin I$. Shrinking $\tilde Y'$ further, we can assume the above holds for all $y\in \tilde Y'$.\\

We now prove that the subset $\{\eta_i(y):i\in I\}\subset X$ (without multiplicities) satisfies the Cayley--Bacharach condition for $K_X$ for $y\in Y$ generic. By passing to an open dense subset of $\tilde Y'$, we may assume that for each pair $\{i,j\}\subset I$, either
\begin{enumerate}
    \item $\eta_i|_{\tilde Y'}=\eta_j|_{\tilde Y'}$, or
    \item $\eta_i(y)\neq \eta_j(y)$ for all $y\in \tilde{Y}'$.
\end{enumerate}
Define a partition $I=B_1\sqcup \dots \sqcup B_l$ into nonempty subsets $B_k$ such that case (1) happens if and only if $i$ and $j$ belong to the same subset $B_k$.\\ 

Now consider any $y\in \tilde Y'$, and let $x_i\coloneqq \eta_i(y)$ for $1\leq i\leq d$. For every subset $B_k$ choose an element $i \in B_k$ and write $z_k=x_i$ and $\psi_k=\eta_i$; this is independent of the choice of $i \in B_k$ by definition. We have the equality of sets $\{z_1,\dots,z_l\}=\{\eta_i(y):i\in I\}$. Consider a form $\omega\in H^0(X,K_X)$. Then the pullback of the form $\sum_{i=1}^d\text{pr}_i(\omega)$ to $\tilde Y'$ vanishes by Theorem~\ref{thm:mumford}. Evaluating this pullback at $y\in \tilde{Y}'$ gives 
$$\sum_{i\in I}\eta_i^*\omega(x_i)=\sum_{i=1}^d\eta_i^*\omega(x_i)=0,$$
and thus
$$\sum_{k=1}^l \left|B_k\right|\psi_k^*\omega(z_k)=0.$$

Assume now that $\omega$ vanishes at all but one of the $z_k$, say, $z_1,\dots,z_{l-1}$. Then it follows that $\left|B_l\right| \psi_l^*\omega(z_l)=0$. Since our ground field is of characteristic zero and $\psi_l^*$ is injective, we have $\omega(z_l)=0$. This proves that $\{z_1,\dots,z_l\}$ satisfies the Cayley--Bacharach condition for $K_X$, as claimed.
\end{proof}
Since we study hypersurfaces in projective space, we will be using the Cayley--Bacharach condition with respect to $\calO(r)$.
\begin{notation}
    We say a collection of points in $\PP^n$ satisfies $\mathit{CB}(r)$ if it satisfies the Cayley--Bacharach conditions with respect to $\calO(r)$.
\end{notation}

We study geometric consequences of condition $\mathit{CB}(r)$ in Section \ref{sec:CB-bootstrap}, which will then be applied in combination with \Cref{cor:rat-equiv-implies-cb}. Here we record two useful results of this type which are already present in the literature.
\begin{lemma}[{Lemma 2.4 of \cite{bcd2014gonality}}]
Let $n\geq 2$ and $Z=\{x_1,\cdots, x_m\}\subset \mathbb{P}^n$ be a collection of distinct points satisfying $\mathit{CB}(r)$, where $r\geq 1$. Then $m\geq r+2,$ and if $m\leq 2r+1,$ then $Z$ is contained in a line.
\label{lem:bcd}
\end{lemma}

\begin{lemma}[Lemma 2.5 of \cite{lopez1995curves}, $\PP^2$ case]
    Let $Z=\{x_1,\dots,x_m\}\subset \PP^2$ be a set of $m\geq 1$ distinct points that satisfy $\mathit{CB}(r)$, $r\geq 1$. Then $m\geq r+2$ and if $k\geq 1$ is an integer satisfying
    \[ m\leq (k+1)r-(k^2-k-1),\quad r\geq 2k-1,\]
    then $Z$ lies on a reduced curve of degree $k$.
    \label{lopez}
\end{lemma}

\section{Cayley--Bacharach bootstrap}\label{sec:CB-bootstrap}

The proof of \Cref{thm:main-introduction} will combine \Cref{cor:rat-equiv-implies-cb} with a method of Riedl and Yang \cite{RY22} to show that for a degree $kd$ point $P$ of $X_{n,d}$ there is a subset of the Galois orbit of $P$ in $X_{n,d}(\C)$ which satisfies $\mathit{CB}(d-2n-2)$. The goal of this section is to explain how this condition, which appears to be rather weak, is in fact sufficient to conclude that $P$ is a transverse intersection of $X_{n,d}$ with a degree $k$ curve, provided that $d \gg n,k$ .\\

Whether a set of points in a projective space satisfies a certain Cayley--Bacharach property is intimately related to the existence of a low-degree curve containing this set. For example, if $2r+1$ or fewer points in projective space satisfy $\mathit{CB}(r)$ then they lie on a line \cite[Lemma 2.4]{bcd2014gonality}. However, this relationship remains conjectural for most ranges of parameters; see \cite[Question 1.1]{picoco2023geometry}, the related \cite[Conjecture 1.2]{levinson2022cayley}, and their partial resolution in \cite{banerjee2024error}. We will prove a new version of this connection in \Cref{lem:high-d-lopez}. By using this result we obtain a small subset of conjugates of $P$ lying on a low-degree curve. We then establish the main result of this section, \Cref{CB-transitive}, which utilizes the Galois action to find a collection of Galois-conjugates of this low-degree curve that cover the whole orbit of $P$ and has total degree $k$.\\ 

The theorems in this section involve a number of parameters: $n,r,d,k$, and various numerical conditions appear as hypotheses. These results will be applied to degree $kd$ points on $X_{n,d}$ when $d\gg k,n$. In this situation $d$ is large, $n$ and $k$ are small, and $r=d-2n-2$, which is less than $d$ but close to it. It is very easy to see that in this regime all of the required inequalities are satisfied.\\

We first prove a simple lemma, which is useful for inductive arguments.

\begin{lemma}
    \label{CB-projections}
Let $S\subset \mathbb{P}^n$, $n\geqslant 3$ be a collection of points and suppose that the image of $S$ under a generic projection $\mathbb{P}^n\dashrightarrow \mathbb{P}^{n-1}$ lies on a degree $k$ curve. Then there is an integer $0\leq k'\leq k$ and a (possibly reducible) curve of degree $k'$ in $\mathbb{P}^n$ containing all but at most $(k-k')(k^2-k')$ points of $S$. 
\end{lemma}
\begin{remark}
    We use the convention that a degree $0$ curve is the empty set. In particular, if $|S|>k^3$ then we must have $k'>0.$
\end{remark}

\begin{proof}
First observe that $S$ must be contained in a (possibly reducible) curve $Z\subset\mathbb{P}^n$ of degree at most $k^2$. Indeed, considering two generic projections $\mathbb{P}^n\dashrightarrow \mathbb{P}^{n-1}$, we see that $S$ is contained in two surfaces $F_1, F_2$ of degree $k$ in $\mathbb{P}^n$ which are cones over degree $k$ curves in $\mathbb{P}^{n-1}$. These surfaces do not share components since the projections are generic. Let $\tilde{F}_1$ be a generic cone over $F_1$ with apex of dimension $n-4$ (so that in $\PP^3$ we have $F_1=\tilde{F}_1$). No component of $F_2$ belongs to $\tilde{F}_1$, and so the intersection $\tilde{F}_1 \cap F_2$ is the desired curve of degree at most $k^2$ that contains $S$.\\

Now consider a third generic projection $\pi: \mathbb{P}^n\dashrightarrow \mathbb{P}^{n-1}$. It is regular on $Z$ and its restriction to every component of $Z$ is birational onto an irreducible curve of the same degree in $\mathbb{P}^{n-1}$. Moreover, we can assume that its restriction to $S$ is injective. Now $\pi(S)$ is contained in both $\pi(Z)$ and a degree $k$ curve $C\subset \mathbb{P}^{n-1}$.\\

We now compare the irreducible components of $\pi(Z)$ and of $C$. Let $D$ be the union of their common irreducible components, $D'$ be the union of the irreducible components of $\pi(Z)$ which are not components of $C$, and $D''$ be the union of the components of $C$ which are not components of $\pi(Z)$. Then $$\pi(S)\subset C\cap \pi(Z)\subset D\cup (D'\cap D'').$$ Let $k'\coloneqq\deg D$, which is at most $k=\deg C$. Since $D'\cap D''$ has at most $(k-k')(k^2-k')$ points, it follows that all but at most $(k-k')(k^2-k')$ points of $S$ are contained in $\pi|_{Z}^{-1}(D)$, which is a (possibly reducible) curve of degree $k'$.
\end{proof}
\begin{remark}
   In \Cref{CB-projections} it is necessary to pass to subsets (namely, allowing $k'<k$). Indeed, consider the set $S'\subset \PP^3$ consisting of $n\gg 0$ collinear points, and let $S$ be the union of $S'$ and $5$ general points in $\PP^3$. Then a general projection of $S$ to $\PP^2$ is contained in a reducible cubic curve since there is a conic through $5$ general points. However, $S$ is not covered by a cubic curve in $\PP^3$ because $5$ general points in $\PP^3$ cannot be covered by a degree $2$ curve. Thus, we cannot have $k'=3$ in this situation. However, note that we may take $k'=1$ or $2$. 
 \end{remark}

We now use \Cref{CB-projections} to construct low-degree curves that cover many points of a set $S$ which satisfies a Cayley--Bacharach condition.

\begin{corollary}\label{lem:high-d-lopez}
    Suppose integers $m,r,k\geq 1$ satisfy the inequalities
    \[ m\leq (k+1)r-(k^2-k-1),\quad r\geq 2k-1.\]
    If $n\geq 2$ and $S=\{x_1,\dots,x_m\}\subset \PP^n$ satisfies $\mathit{CB}(r)$, then there exists an integer $0\leq k'\leq k$ such that at least $m-(k-k')(k^2-k')$ points of $S$ lie on a reduced degree $k'$ curve. 
\end{corollary}
\begin{proof}
    For $n=2$, we may take $k'=k$ by \Cref{lopez}. For $n>2$, we proceed by induction on $n$; note that the property $\mathit{CB}(r)$ is preserved under generic projections.\\

    By the induction hypothesis, there is an integer $0\leq k' \leq k$ and a collection $S'$ of $m-(k-k')(k^2-k')$ points of $S$ such that a generic projection of $S'$ lies on a curve of degree $k'$. By \Cref{CB-projections}, we can find $0\leq k'' \leqslant k'$ and a degree $k''$ curve containing at least 
    \begin{align*}
        &m-(k-k')(k^2-k')-(k'-k'')(k'^2-k'')\\ &\;\;\geq m-(k-k')(k^2-k'')-(k'-k'')(k^2-k'')\\
        &\;\;=m-(k-k'')(k^2-k'')
    \end{align*}
    points of $S'$. This concludes the induction step.
\end{proof}
\begin{remark}
    If $k'>0$, by passing to irreducible components, there exists an integer $0<k''\leq k'$ such that at least
    \[ \ceil*{\frac{k''}{k'}\left(m-(k-k')(k^2-k')\right)}\]
    points of $S$ lie on an irreducible degree $k''$ curve. This follows from a pigeonhole argument. Let $C$ be the degree $k'\geq 1$ curve passing through at least $m':=m-(k-k')(k^2-k')$ points of $S$. Denote by $C_1,\dots,C_h$ the components of $C$ and write $\deg(C_i)=k_i\geq 1$, where $\sum_{i} k_i=k'$. If for all $1\leq i\leq h$,
    \[|C_i\cap S|<\frac{k_i}{k'}m',\]
    then
    \[ |C\cap S|\leq \sum_i |C_i \cap S|<\frac{\sum_i k_i}{k'}m'=m',\]
    a contradiction. This shows there exists $i$ such that
      \[|C_i\cap S|\geq\ceil*{\frac{k_i}{k'}m'}.\]
  \end{remark}

Before starting the proof of \Cref{CB-transitive} we collect a minimal set of assumptions on $d,r,k$ which reflect the geometric setting of ``low degree points on high degree hypersurface in low dimensional projective space''. For this we introduce the following condition $(\dagger)$ on the integers $d,r,k\geq 1$:

    \begin{enumerate}[label=\textup{($\dagger_{\text{\roman*}}$)}]
        \item\label{dagger1} $d\geq r\geq 2k^2-1;$
        \item\label{dagger2} $d\geq 2k(d-r)+3k^3-k^2-4k+1.$
    \end{enumerate}

Condition $(\dagger)$ gives a minimalistic set of inequalities; since we will need to appeal to their various corollaries we collect all of them in the following cumbersome but entirely elementary lemma.

\begin{lemma}\label{lem:inequalities}
    Suppose integers $d,r,k\geq 1$ satisfy $(\dagger)$. Then
    \begin{enumerate}[label=\textup{(\arabic*)}]
        \item $r\geq 2k-1$;
        \item $kd\leq (k+1)r-(k^2-k-1)$;
        \item For every integer $m'$ with $r+2\leq m'\leq kd$, let
        \[ k'\coloneqq \ceil*{\frac{m'+k^2-k-1}{r}}-1,\]
        and for all integers $k'',k'''$ with $1\leq k'''\leq k''\leq k'$, let
        \[m''\coloneqq\ceil*{\frac{k'''}{k''}(m'-(k'-k'')(k'^2-k''))}\]
        and
        \[l\coloneqq \floor{k/k'''}.\]
        Then
        \begin{enumerate}[label=\textup{(\alph*)}]
            \item $(l+1)m''-\binom{l+1}{2}k'''^2-kd>0$;
             \item $m'>k'^3$;
            \item $m''>kk'''$.
        \end{enumerate}
    \end{enumerate}
\end{lemma}

\begin{proof}
    The inequality (1) follows from \ref{dagger1}, and (2) is equivalent to  
    $$d\geq (k+1)(d-r)+(k^2-k-1),$$
    which follows from \ref{dagger2} and the $d-r\geq 0$ part of \ref{dagger1}.\\
    
    To verify (3a), let $N=d-r\geq 0$ (recall \ref{dagger1}), and let $m',k'',k'''$ be integers satisfying the hypotheses of (3). By (2), we have $1\leq k'\leq k$. Using
    \[m'>rk'-(k^2-k-1),\quad k\leq (l+1)k'''-1,\]
    we get
    \begin{align*}
        &(l+1)m''-\binom{l+1}{2}k'''^2-kd \\&\geq (l+1)\frac{k'''}{k''}\left(m'-(k'-k'')(k'^2-k'')\right)-\binom{l+1}{2}k'''^2-kd  \\
        &> (l+1)\frac{k'''}{k'}\left((d-N)k'-(k^2-k-1)-(k'-k'')(k'^2-k'')\right)\\
        &\qquad-\binom{l+1}{2}k'''^2-((l+1)k'''-1)d \\
        &=d-(l+1)k'''N-A,
    \end{align*}
    where
    \begin{align*}
        A&\coloneqq(l+1)\frac{k'''}{k'}\left((k^2-k-1)+(k'-k'')(k'^2-k'')\right)+\binom{l+1}{2}k'''^2.
    \end{align*}
    Hence, (3a) is verified if we can prove $d\geq (l+1)k'''N+A$ for all choices of $m',k'',k'''$. We now find an upper bound for the right-hand side. First, recall that $lk'''\leq k$, so 
    \[ (l+1)k'''\leq k+k'''\leq 2k.\]
    Second, using $1\leq k'''\leq k''\leq k$ and $lk'''\leq k$, and the fact that $k^2-k-1\geq 0$ if $k\geq 2$, we get
    \begin{align*}
        A&\leq \frac{k+k'''}{k'}\left(k^2-k-1+(k'-1)(k'^2-1)\right)+\frac{(l^2+l)k'''^2}{2}\\
        &\leq\left(\frac{k}{k'}+1\right)(k^2-k-1)+(k+k''')\left(1-\frac{1}{k'}\right)(k'^2-1) +\frac{k^2+kk'''}{2}\\
        &\leq (k+1)(k^2-k-1)+2k\left(1-\frac{1}{k}\right)(k^2-1)+\frac{k^2+k^2}{2}\\
        &=3k^3-k^2-4k+1.
    \end{align*}
    As a result, \ref{dagger2} implies (3a).\\

 To verify (3b), we compute
    \begin{align*}
        m'-k'^3&>(rk'-(k^2-k-1))-k'^3
        \\&= k'\left(r-\left(k'^2+\frac{k^2-k-1}{k'}\right)\right).
    \end{align*}
    Again, this is greater than $0$ if \ref{dagger1} is satisfied.\\

    To verify (3c), we compute
    \begin{align*}
        m''-kk'''&\geq \frac{k'''}{k''}(m'-k'^3)-kk'''\\
        &> k'''\left(\frac{(rk'-(k^2-k-1))-k'^3}{k'}-k\right) \\
        &=k'''\left(r-\left(\frac{k^2-k-1}{k'}+k'^2+k\right)\right).
    \end{align*}
    The last expression is greater than or equal to $0$ if \ref{dagger1} is satisfied.
\end{proof}

\begin{theorem}\label{CB-transitive}
    Suppose $K$ is a field of characteristic zero, $P$ is a closed point of $\PP^n_K$ of degree $kd$ for some integers $d,k\geq 1.$ Let $S$ be the Galois orbit of $P$ in $\PP^n(\Kbar)$. If a subset $S' \subset S$ satisfies $\mathit{CB}(r)$ for some $1\leq r\leq d$ satisfying
    \[r\geq 2k^2-1, \quad d\geq 2k(d-r)+3k^3-k^2-4k+1,\]
    then $S$ can be covered by a reduced curve $C$ of degree at most $k$. Moreover, we can choose such a $C$ which is defined over $K$.
\end{theorem}
\begin{proof}
    Our assumptions are precisely $(\dagger)$, so $d,r,k$ also satisfy the conclusions of \Cref{lem:inequalities}.\\
    
    We first use \Cref{lem:high-d-lopez} to find a curve $D$ of controlled degree covering a subset of $S'$. Namely, let $m'\coloneqq|S'|\leq kd$, and
    \[ k'\coloneqq\ceil*{\frac{m'+k^2-k-1}{r}}-1.\]
    Note that $m'\geqslant r+2$ by \Cref{lem:bcd}, $k'\leq k$ by \Cref{lem:inequalities}(1), and $m' > k'^3$ by Lemma \ref{lem:inequalities}(3b).  By (1), the definition of $k'$, and \Cref{lem:high-d-lopez} applied to $S'$ and parameters $m', r, k'$, there exists an integer $0\leq k''\leq k'$ such that at least $m'-(k'-k'')(k'^2-k'')$ points of $S'$ lie on a degree $k''$ curve. Since $m'-(k'-k'')(k'^2-k'')\geq m'-k'^3>0$, we must have $k''\geq 1$. By passing to an irreducible component, there is an integer $1\leq k'''\leq k''$ and an irreducible degree $k'''$ curve $D$ that passes through at least
    \[m''\coloneqq \ceil*{\frac{k'''}{k''}(m'-(k'-k'')(k'^2-k''))}\]
    points of $S'$.\\

    Let $l=\floor{k/k'''}$. We claim that $S$ can be covered by at most $l$ Galois conjugates of $D$, so if we denote their union by $C$, we will have $\deg C\leq k$ as required. To construct $C$, start with $C_1=D$ and take any point in $S$ not contained in $D$. Using transitivity, we can find a Galois conjugate of $D$ through this point and consider the union $C_2$ of $C_1=D$ and this Galois conjugate. We can proceed inductively until every point in $S$ is covered by some curve $C_N$.\\
    
    Suppose  the claim is false, then $C_{l}$ fails to cover $S$, and so $N\geqslant l+1$. Consider the curve $C_{l+1}$. Let's estimate from below the number of points of $S$ covered by $C_{l+1}$. Each conjugate of $D$ contains at least $m''$ points of $S$, and each pair of two conjugates of $D$ intersect in at most $k'''^2$ points since $D$ is irreducible. By inclusion-exclusion, we conclude that 
    \[|C_{l+1}(\overline{K})\cap S|\geqslant (l+1)m'' - \binom{l+1}{2}k'''^2.\]
    But (3a) states that 
    \[(l+1)m'' - \binom{l+1}{2}k'''^2 > |S|=kd,\]
   which provides a contradiction.\\

    Now let $C$ be any output of the construction above. We claim that $C$ is defined over $K$. Suppose not, then there exists an element $g\in \Gal(\overline{K}/K)$ such that $gD$ is not a component of $C$. Then
    \[m''\leq |gS\cap gD|=|S\cap gD|\leq |C\cap gD|\leq kk'''.\]
    But this contradicts (3c).
\end{proof}

\begin{remark}
We will apply \Cref{CB-transitive} to a degree $kd$ point on the universal degree $d$ hypersurface $X_d$. In this case, since $X_d$ cannot contain $K$-rational curves of degree less than $d$, the curve $X$ of \Cref{CB-transitive} is necessarily a degree $k$ curve intersecting $X_d$ transversely.
\end{remark}

\section{Proof of \Cref{thm:main-introduction}}

We begin by proving \Cref{projnormCB}. Following Riedl and Yang \cite{RY22}, a \emph{parametrized} \emph{$r$-plane} \emph{in $\PP^m$} is an injective linear map $\Lambda: \PP^r \to \PP^m$. 
 
Given a pointed hypersurface $(p, X)\in \calX_{m-1,d}$ in $\PP^m$ and $\Lambda: \PP^r \to \PP^m$, a parametrized $r$-plane whose image passes through $p$ and does not lie entirely in $X$, we say that the pair $(\Lambda^{-1}(p),\Lambda^{-1}(X))\in \calX_{r-1,d}$ is a \emph{parametrized $r$-plane section} of $(p,X)$. Together with the Mumford--Ro\u \i tman theorem on zero-cycles (or rather \Cref{corCB}), the main result we will use is the following:

\begin{theorem}[{Theorem 2.3 of \cite{RY22}}]\label{Riedl-Yang}
Fix a positive integer $n_0$ and consider for all $1\leq m\leq n_0$ loci $Z_{m,d}\subset \calX_{m,d}$ which are countable unions of locally closed subvarieties and satisfy:
\begin{enumerate}[label=\textup{(\arabic*)}]
\item If $(p,X)\in Z_{m,d}$ is a parametrized hyperplane section of $(p',X')\in \calX_{m+1,d}$, then $(p',X')\in Z_{m+1,d}$.
\item $Z_{n_0-1,d}$ has codimension at least one in $\calX_{n_0-1,d}$. 
\end{enumerate}
Then the codimension of $Z_{n_0-c,d}$ in $\calX_{n_0-c,d}$ is at least $c$.
\end{theorem}

We begin by proving \Cref{projnormCB} which we restate here for the reader's convenience.

\begin{proposition}[= \Cref{projnormCB}]
Let $X\subset\mathbb{P}^{n+1}_\C$ be a very general hypersurface of degree $d\geq 2n+3$. If an effective zero-cycle $p_1+\dots+p_{kd}$ of degree $kd$ on $X$ is rationally equivalent to $k\cdot c_1(\mathcal{O}_X(1))^{n}$ then a subset of $\{p_1,\dots,p_{kd}\}$ satisfies $\mathit{CB}(d-2n-2)$.\label{projnormCB-restated}
\end{proposition}

\begin{proof}
For $r,m\geq 1$, let
\begin{align*}
Z_{m,d}^{(r)}  &\coloneqq  \Bigg  \{ (p_1,X) \in \mathcal{X}_{m,d}  
\Bigg | \parbox{2.7in}{\begin{center} $\exists p_2,\ldots, p_{kd}\in X$ such that $\sum_{i=1}^{kd} p_i\sim_{\text{rat}}k c_1(\mathcal{O}(1))^{m}$ and no subset of $p_1,\ldots, p_{kd}$ satisfies $\mathit{CB}(r)$\end{center}}  
\Bigg  \}.
\end{align*}
We want to apply \Cref{Riedl-Yang} to the sets $Z_{m,d}^{(r)}$ so we verify some of the hypotheses of the theorem.\\

\begin{lemma}
The subsets $Z_{m,d}^{(r)}$ are countable unions of locally closed subvarieties and they satisfy condition $(1)$ of \Cref{Riedl-Yang}.
\end{lemma}
\begin{proof}
It is easy to show (see for instance \cite{roitman1971gamma}) that given a smooth hypersurface $X\subset \mathbb{P}^{m+1}$ of degree $d$ and $k\in \mathbb{Z}_{>0}$ the following subset is a countable union of locally closed subsets:
\begin{align*}
W_1=\{(x_1,\dots, x_{kd}): x_1+\dots+x_{kd}\sim_{\textup{rat}}\mathcal{O}_X(1)^m\}\subset X^{kd}.
\end{align*}
We will show that the following subset is constructible.
\begin{align*}
W_2=\{(x_1,\dots, x_{kd}): \textup{ a subset of } {x_1,\dots, x_{kd}} \textup{ satisfies }\mathit{CB}(r)\}\subset X^{kd}.
\end{align*}
 Of course this implies that $Z_{m,d}^{(r)}=\text{pr}_1(W_1\setminus W_2)$ is a countable union of locally closed subsets.\\

First, to show that $W_2$ is constructible it suffices to show that for any positive integer $i$, the following subset is constructible
$$\{(x_1,\dots, x_i): \{x_1,\dots, x_i\}\text{ does not satisfy }\mathit{CB}(r)\}\subset (\mathbb{P}^{m+1})^i.$$
Of course, it suffices to show that the following subset is constructible
$$S\coloneqq \{(x_1,\dots, x_i): H^0(\mathcal{O}_{\mathbb{P}^n}(r)\otimes \mathcal{I}_{x_1,\dots, x_{i-1}})\neq  H^0(\mathcal{O}_{\mathbb{P}^n}(r)\otimes \mathcal{I}_{x_1,\dots, x_i})\}\subset (\mathbb{P}^{m+1})^i.$$
 Let $\mathcal{Y}\subset \mathbb{P}^{m+1}\times |\mathcal{O}_{\mathbb{P}^{m+1}}(r)|$ be the universal hypersurface of degree $r$ and consider the fiber products $$F_i\coloneqq \mathcal{Y}^{i}_{|\mathcal{O}_{\mathbb{P}^{m+1}}(r)|}\subset (\mathbb{P}^{m+1})^{i}\times |\mathcal{O}_{\mathbb{P}^{m+1}}(r)|,$$
 and
 $$F_{i-1}\coloneqq \mathcal{Y}^{i-1}_{|\mathcal{O}_{\mathbb{P}^{m+1}}(r)|}\subset (\mathbb{P}^{m+1})^{i-1}\times |\mathcal{O}_{\mathbb{P}^{m+1}}(r)|.$$
Consider the following pullback diagram
 \[
 \begin{tikzcd}[column sep=large]
F_i' \ar[r] \ar[d] & F_{i-1}\ar[d,"\phi_i"]\\
(\mathbb{P}^{m+1})^{i}\ar[r,"\text{pr}_{1\dots (i-1)}"] & (\mathbb{P}^{m+1})^{i-1}.
 \end{tikzcd}
 \]
 We can think of $F_i'$ as a closed subset of $(\mathbb{P}^{m+1})^{i}\times |\mathcal{O}_{\mathbb{P}^{m+1}}(r)|$. Now the fiber of $F_i$ over $(x_1,\dots, x_i)$ is the linear subspace of $|\mathcal{O}_{\mathbb{P}^{m+1}}(r)|$ consisting of polynomials vanishing at $x_1,\dots, x_i$. The fiber of $F_i'$ over the same point is the linear subspace of $|\mathcal{O}_{\mathbb{P}^{m+1}}(r)|$ consisting of polynomials vanishing at $x_1,\dots, x_{i-1}$. We deduce that $$S=\text{pr}_{(\mathbb{P}^{m+1})^{i}}(F_i'\setminus F_i).$$
 Chevalley's theorem then implies that $S$ is constructible. 
 \renewcommand{\qedsymbol}{$\diamondsuit$}
\end{proof}

We continue the proof of \Cref{projnormCB-restated}. Fix $n,d,k$ with $d\geq 2n+3$, and let $n_0=2n+1$ and $r=d-2n-2\geq 1$. We shall show that the subsets $Z^{(r)}_{m,d}$ for $1\leq m\leq n_0$ satisfy the assumption (2) of \Cref{Riedl-Yang}, namely $Z_{2n,d}^{(d-2n-2)}\subset \mathcal{X}_{2n,d}$ has codimension at least one.\\

We first present the end of the proof assuming this. \Cref{Riedl-Yang} then implies that $$Z_{n,d}^{(d-2n-2)}\subset \mathcal{X}_{n,d}$$
has codimension at least $(n+1)$, and thus for a very general degree $d$ hypersurface $X\subset \mathbb{P}^{n+1}$, the fiber of $Z^{(d-2n-2)}_{n,d}$ over $[X]\in B_{n,d}$ is empty. It follows that every effective zero-cycle of degree $kd$ rationally equivalent to $kc_1(\mathcal{O}(1))^{n}$ has a subset that satisfies $\mathit{CB}(d-2n-2)$, as claimed.\\

We now prove the claim by contradiction. If the claim does not hold, given a generic hypersurface $X$ of degree $d$ in $\mathbb{P}^{2n+1}$, there is a locally closed subvariety $Y\subset \text{Sym}^{kd}X$ of dimension $2n$ such that
\begin{enumerate}
\item the restriction of the map $\text{Sym}^{kd}X\longrightarrow \text{CH}_0(X)$ to $Y$ is constant,
\item for a general $p_1\in X$ there are points $p_2,\dots, p_{kd}$ such that $p_1+\cdots+p_{kd}\in Y$,
\item a general point of $Y$ parametrizes an effective zero-cycle of degree $kd$ whose support does not contain a subset satisfying $\mathit{CB}(d-2n-2)$.
\end{enumerate}

By (1)--(2), the assumptions of \Cref{corCB} are satisfied. Thus a generic point of $Y$ parametrizes an effective zero-cycle of degree $kd$ whose support contains distinct points $x_1,\dots,x_m\in X$ with $m\leq d$ satisfying the Cayley--Bacharach condition for $K_X$. However, $K_X=\calO(d-2n-2)$, violating assumption (3).
\end{proof}

We are now ready to prove our main theorem for effective cycles of dimension zero.

\begin{theorem}[$\subset$ \Cref{thm:main-introduction}]\label{thm:main-0cyc}
    Suppose $d,n,k$ are integers satisfying 
    \[ d\geq 4 k n + 3 k^3 - k^2 + 1. \]
    Then every degree $kd$ point on $X_{n,d}$ is contained in a degree $k$ curve not lying on $X_{n,d}$. 
\end{theorem}

\begin{proof}
Given a Galois orbit of $kd$ distinct points $p_1,\ldots,p_{kd} \in X_{n,d}(\overline{K})$, the zero-cycle $p_1+\cdots+p_{kd}$ is rationally equivalent to $kc_1(\mathcal{O}(1))^{n}$ by \Cref{lemma:chow-group-calculation}. By \Cref{projnormCB}, there is a subset of the $p_i$'s which satisfies $\mathit{CB}(d-2n-2)$.\\

Our numerical assumption $d\geq 4 k n + 3 k^3 - k^2 + 1$ is equivalent to the second hypothesis of \Cref{CB-transitive} with $r=d-2n-2$. It also implies the first hypothesis $d\geq 2n+2+2k^2-1$, as can be seen from the simple calculation
\[(4 k n + 3 k^3 - k^2 + 1)-(2n+2+2k^2-1)=(3 k^2 + 2 n) (k - 1) + 2 k n \geq 0.\]
We then conclude that there is a degree $k$ curve $C$, defined over $K$, which contains $p_1, \dots, p_{kd}$. But a degree $k$ curve always has degree $k$ points, and so $C \not\subset X_{n,d}$ since $k<d$. 
\end{proof}

\section{Effective cycles}\label{sec:cycles}

To prove \Cref{thm:main-introduction} for cycles of all dimensions we appeal to the following    special case of a theorem of Chiantini and Ciliberto.

\begin{theorem}[Chiantini--Ciliberto \cite{chi-cil}]\label{chi-cil-lifting}
       Let $n,d,k$ be positive integers satisfying $d> 2(n+k)$. Suppose $Z \subset \PP^{n+1}_\C$ is a degree $kd$ subvariety such that for a general linear subspace $\Lambda$ of dimension $n+1-\dim Z$ the intersection $\Lambda \cap Z$ lies on a degree $k$ irreducible curve. Then $Z$ belongs to a unique degree $k$ subvariety $W$ of $\PP^{n+1}$ of dimension $\dim Z + 1$. 
\end{theorem}
\begin{proof}
    The existence of $W$ follows from \cite[Theorem 0.2]{chi-cil} with ${\bf X}=Z,$ ${\bf n}=\dim Z,$ ${\bf d} = kd$, ${\bf h}=n+1-\dim Z$, ${\bf s}=k$, and ${\bf r}=n+1$ (the letters in bold correspond to the notation used in \cite{chi-cil}). Note that in \cite{chi-cil} the word ``variety'' is used for a possibly reducible scheme of finite type over $\C$; in particular ${\bf X}$ is \emph{not} assumed to be irreducible. We remark also that our numerical condition $d>2(n+k)$ implies the numerical condition required by \cite[Theorem 0.2]{chi-cil}, namely 
    \[kd> k(2n-1-\dim Z)+m(m-1)(n-\dim Z)+2mr-2,\]
    where $m$ and $r$ are the integral part and the remainder of the fraction $(k-1)/(n-\dim Z)$. Finally, note that since $d>k$, the set $\Lambda \cap Z $ of $kd$ points cannot lie on two different irreducible curves of degree $k$. Therefore \cite[Theorem 0.2]{chi-cil} shows that $W$ is unique.
\end{proof}
\begin{remark}
 In the setting of Theorem \ref{chi-cil-lifting}, if $Z$ is irreducible, then the monodromy of the intersection $\Lambda \cap Z$ is the full symmetric group (as $\Lambda$ varies). More precisely, consider the incidence correspondence 
 \[I =\{(x, \Lambda) \colon x \in \Lambda\}\subset Z \times \Gr(n+1-\dim Z, n+1).\]
 The \emph{sectional monodromy group} is the monodromy of the generically finite degree $\deg Z$ cover $I \longrightarrow \Gr(n+1-\dim Z, n+1)$. It is the full symmetric group $S_{\deg Z}$ since it is transitive (since $I$ is irreducible) and contains a transposition, which can be seen by choosing a simply tangent plane $\Lambda$. Similarly, when $Z$ is not irreducible, since one can choose $\Lambda$ tangent to one irreducible component and intersecting the rest transversely, the monodromy of $\Lambda \cap Z$ is a product of symmetric groups. This provides a kind of ``uniform position principle'' for the generic intersection $\Lambda \cap Z$, which we will use later (compare with the classical case of curves \cite[Chapter 3 \S 1]{ACGH} ) 
\end{remark}

To apply \Cref{chi-cil-lifting} to our situation, we need to work around the hypothesis of irreducibility over $\C$ of the interpolating curve. First, we have the following lemma.

\begin{lemma}\label{lemirr}
   Let $n,k,d$ be positive integers such that $d \geqslant 2$. Consider $Z\subset \PP^n_\C$ an equidimensional variety of degree $kd$, and write $Z_1,\ldots, Z_s \subset Z$ for its irreducible components. Suppose that for a general linear subspace $\Lambda$ of dimension $n-\dim Z + 1$ the intersection $\Lambda \cap Z$ belongs to a curve $\Gamma$ of degree $k$. Then there is a collection $\Gamma_1, \dots, \Gamma_m$ of irreducible components of $\Gamma$  such that
    \begin{enumerate}[label=\textup{(\arabic*)}]
        \item\label{cond1} $\Gamma_i$ contains $d \deg \Gamma_i$ points of $Z \cap \Lambda$;
        \item\label{cond2} For all $i,j$ either $\Gamma_i$ contains $Z_j \cap \Lambda$ or $\Gamma_i \cap (Z_j \cap \Lambda)=\emptyset$;
        \item\label{cond3} $\bigcup_i \Gamma_i$ contains $\Lambda \cap Z$.
    \end{enumerate} 
\end{lemma}

\begin{proof}
  There exists an irreducible component $\Gamma_1$ of $\Gamma$ which covers at least $d\deg \Gamma_1$ points of $Z \cap \Lambda$. Suppose that for some $Z_j$ the intersection $\Gamma_1 \cap Z_j$ is nonempty but $\Gamma_1 \not\supset (Z_j \cap \Lambda)$. Then choose a point $P \in \Gamma_1 \cap (Z_j\cap \Lambda)$ and $Q\in (Z_j \cap \Lambda) \setminus \Gamma_1$. Let $g$ be an element of the sectional monodromy group which switches $P$ with $Q$. Then $g\Gamma_1 \neq \Gamma_1$ and $g\Gamma_1 \cap \Gamma_1$ has at least $d \deg \Gamma_1 -1$ points, but both curves are irreducible of degree $\deg \Gamma_1 \leqslant k$. This contradicts the assumption $k<d$. Thus there is a subset $J\subset \{1,\ldots, s\}$ such that $(Z_j \cap \Lambda)\subset \Gamma_1$ for all $j\in J$ and $\sum_{j\in J}\deg Z_j=d\deg \Gamma_1$. Replacing $Z$ with the degree $(k-\deg \Gamma_1)d$ variety $Z \setminus \bigcup_{j\in J} Z_i$ and $\Gamma$ with the degree $(k-\deg \Gamma_1)$ curve $\Gamma\setminus \Gamma_1$, and applying the same argument repeatedly, we get the desired collection $\Gamma_1, \dots, \Gamma_m$.\qedhere
\end{proof}

We can now prove our main theorem for cycles of any dimension.

\begin{theorem}[{$=$ Theorem \ref{thm:main-introduction}}]
    Suppose $d,n,k$ are integers satisfying 
      \[ d\geq 4 k n + 3 k^3 - k^2 + 1.  \]
 Then every degree $kd$ effective cycle on $X_{n,d}$ is contained in a degree $k$ subvariety none of whose components lie on $X_{n,d}$.
\end{theorem}
\begin{proof}
    For zero-cycles, this is precisely \Cref{thm:main-0cyc}. Now suppose that $Z \subset X_{n,d}$ is an irreducible subvariety of degree $kd$. Choose an embedding of $\overline{K}$ into $\mathbb{C}$ and write $Z_1, \ldots, Z_s$ for the irreducible components of $Z_\mathbb{C}$. Given a generic linear subspace $\Lambda\subset \mathbb{P}^{n+1}_K$ of dimension $(n+1-\dim Z)$, note that $\Lambda_{\mathbb{C}}\cap Z_\mathbb{C}$ is contained in a (possibly reducible) degree $k$ curve $\Gamma$ by \Cref{thm:main-0cyc}. Applying \Cref{lemirr}, we can assume that $\Gamma$ contains a union of irreducible components $\Gamma_1, \dots, \Gamma_m$ satisfying conditions \ref{cond1}--\ref{cond3}. Suppose there is another union $\Gamma_1', \dots, \Gamma_{m'}'$ satisfying \ref{cond1}--\ref{cond3} and, without loss of generality, that $\Gamma_1 \neq \Gamma_i'$ for all $i$. Then the intersection $\Gamma_1 \cap \left(\bigcup \Gamma_i'\right)$ on the one hand contains $d \deg \Gamma_1$ points of $Z \cap \Lambda$, and on the other hand has at most $\deg \Gamma_1 \cdot (\sum \deg \Gamma_i')\leqslant k\deg \Gamma_1$ points. Since $d>k$, this is a contradiction. Thus the union $\bigcup \Gamma_i$ satisfying \ref{cond1}--\ref{cond3} is unique, and  we write $C=\bigcup \Gamma_i$.\\
    
    Applying \Cref{lemirr}, we find irreducible subvarieties $\gamma_1,\cdots, \gamma_m\subset \mathbb{P}^{n+1}_\mathbb{C}\times \text{Gr}(n+2-\dim Z, n+2)$ whose fiber over a generic $\Lambda\in \text{Gr}(n+2-\dim Z, n+2)$ is an irreducible curve $\Gamma_i$ satisfying \ref{cond1}--\ref{cond3}. Note that this generic fiber is irreducible since the curves $\Gamma_i$ are determined by $\Lambda$ and the irreducible components $Z_j$ such that $(Z_j\cap C)\subset \Gamma_i$. In particular, the monodromy group obtained by varying $\Lambda$ acts trivially on the set $\{\Gamma_1,\ldots, \Gamma_m\}$. Now fix an $i$ and consider the variety $Z[i]:=\text{pr}_1(\gamma_i\cap (Z\times \text{Gr}(n+2-\dim Z, n+2)))$ which consists of the union of the $Z_j$ such that $I\cap ( Z_j \times \text{Gr}(n+2-\dim Z, n+2))\subset \gamma_i$, where $I=\{(x,\Lambda): x\in \Lambda\}\subset \mathbb{P}^{n+1}_\mathbb{C}\times \text{Gr}(n+2-\dim Z, n+2)$ is the incidence correspondence. The variety $Z[i]$ is a priori not irreducible but for a generic $\Lambda\in \text{Gr}(n+2-\dim Z, n+2)$, the intersection $Z[i]\cap \Lambda$ is contained in the irreducible curve $\gamma_{i,\Lambda}$.\\

    Finally, we can apply \Cref{chi-cil-lifting} to deduce that $Z[i]$ is contained in a unique degree $\deg \Gamma_i$ subvariety of $\mathbb{P}^{n+1}_\mathbb{C}$ of dimension $\dim Z+1$. Doing this for all $i$, we see that $Z_\mathbb{C}$ is in fact contained in a unique (possibly reducible) subvariety $\bigcup_{i=1}^m Z[i]$ of $\mathbb{P}^{n+1}_{\mathbb{C}}$ of degree at most $k$ and dimension $\dim Z+1$. Since this variety is unique it descends to a variety over $K$.
\end{proof}

\section{Ubiquity of points on unirational varieties}

In this section we give examples illustrating that some degree assumptions are necessary in order for all equidimensional subvarieties $Z\subset X_{n,d}$ of degree $kd$ to be contained in degree $k$ subvarieties of dimension $\dim Z+1$. We focus on the case $\dim Z=0$, $d=3$, and $k=1$, showing that most degree $3$ points on cubic hypersurfaces of dimension at least $2$ are not collinear. This is a manifestation of the following much more general result.

\begin{theorem}\label{thm-unirationals}
Suppose $k$ is an \emph{infinite} field, $X/k$ is a smooth irreducible variety, and $L/k$ is a separable degree $d$ field extension such that $X_L$ is unirational over $L$. Then the set of closed points of $X$ with residue field $L$ is dense when viewed as a subset of $\Sym^d X$.
\end{theorem}
\begin{proof}
Without loss of generality, assume that $X$ is an affine variety. Consider the Weil restriction of scalars $\Res_{L/k} X_L$ (see \cite[\S 4.6]{poonen2023rational}), and the associated map $$\phi\colon \Res_{L/k} X_L \longrightarrow \Sym^d X$$ defined on $\overline{k}$-points by $$\phi\left((x_\sigma)_{\sigma\in \Hom_{\text{$k$-$\mathrm{Alg}$}}(L,\overline{k})}\right)=\sum_{\sigma\in \Hom_{\text{$k$-$\mathrm{Alg}$}}(L,\overline{k})} x_\sigma, \quad x_\sigma\in X(\overline k).$$
Note that it is a morphism defined over $k$ and that it is surjective.\\

We claim that $k$-points in $\Res_{L/k} X_L$ are dense, as we now prove. Since $X_L$ is unirational, there is a dominant rational map $\mathbb{A}^n_L\dashrightarrow X_L$ for $n=\dim X$, which induces a dominant rational map $\Res_{L/k} \Aff^{n}_L \dashrightarrow \Res_{L/k} X_L$. However, $\Res_{L/k}\Aff^{n}_L \simeq_k \Aff^{dn}$, and $k$-points are dense in this variety since $k$ is infinite.\\

Since $\phi$ is surjective, the image of $(\Res_{L/k} X_L)(k)$ is dense in $\Sym^d X$. This means that the set of closed points $P$ of $X$ with residue field a subextension of $L/k$ is dense in $\Sym^d X$. However, if $P\in X(L)$ is defined over a \emph{proper} subextension of $L/k$, then $\phi(P)$ is in the big diagonal of $\Sym^d X$. 
This shows that the set of closed points of $X$ with residue field precisely $L$ is still dense in $\Sym^d X$.
\end{proof}

\begin{corollary}\label{cor:cubic-on-cubic}
    Suppose $k$ is an infinite field and $X$ is a degree $3$ smooth hypersurface in $\PP^{n+1}_k$, $n\geqslant 2$. Then degree $3$ points of $X$ are dense when viewed as a subset of $\Sym^3X$. In particular, degree $3$ points on a universal cubic hypersurface are not all collinear.
\end{corollary}
\begin{proof}
    A classical theorem of Segre \cite{segre1943note}, generalized by Koll\'ar to arbitrary dimension \cite{kollar2002unirationality}, says that for $n\geq 2$ and a field $L$ a cubic hypersurface in $\PP^{n+1}_L$ is unirational if and only if it admits an $L$-point. Intersecting $X$ with a line over $k$ shows that $X$ has a rational or cubic point. In either case, there is a cubic extension $L/k$ such that $X(L)\neq \emptyset$. Applying Theorem \ref{thm-unirationals} concludes the proof of the first assertion. To prove the second assertion, note that the set of 3-tuples of collinear points on $X$ is a closed subset of $\Sym^3 X$ of codimension $n>0$.  
\end{proof}

Theorem \ref{thm-unirationals} can be applied to some arithmetic questions on the sets of degree $d$ points on varieties (see, for instance, \cite{viray-vogt}). In particular, it demonstrates ubiquity of degree $d$ points on del Pezzo surfaces. 
\begin{corollary}
    Suppose $S$ is a del Pezzo surface of degree $e\geqslant 3$. Then for any integer $d$, if $X$ admits a degree $d$ point, then degree $d$ points are dense in $\Sym^d S$.  
\end{corollary}
\begin{proof}
    Del Pezzo surfaces of degree at least $3$ are unirational if and only if they admit a rational point; see, for example \cite[Remark 9.4.11]{poonen2023rational}. The claim now follows from Theorem \ref{thm-unirationals}.
\end{proof}

\section{Open problems}
In \cite[Problem 3.5]{Farb}, Farb asks for a classification of rational multisections of $\mathcal{X}_{n,d}\longrightarrow B_{n,d}$, or in other words, algebraic points of $X_{n,d}$. The structure of this set of algebraic points appears to be too chaotic to hope for a classification even when $d$ is large relative to $n$. In particular, whenever $X_{n,d}$ is not a curve, we expect that many multisections do not arise as complete intersections. The same should also hold for effective cycles of higher dimension and codimension at least $2$. An explicit example of such behavior occurs for degree $4$ points on a quadric:
\begin{example}
    Consider a degree $2$ point $P$ on the universal quadratic surface $X_{2,2}$, and let $C$ denote a twisted cubic that contains $P$ and is not contained in $X_{2,2}$. The residual intersection $Q=(C\cap X_{2,2})\setminus P$ is in general a nonplanar degree $4$ point on $X_{2,2}$ with Galois group $S_4$. Such a point cannot be supported on a degree $2$ curve.
\end{example}

Similarly, there are general residuation constructions that can be used to produce degree $kd$ points on $X_{n,d}$, which are most likely not supported on a curve of degree $k$. However, proving this rigorously is a nontrivial task, since the family of all degree $k$ curves in $\PP^{n+1}$ becomes unwieldy for large $k$. In particular, we expect that when $n$ and $d$ are fixed and $k$ is large, points of degree $kd$ on $X_{n,d}$ are not all supported on degree $k$ curves:

\begin{question}\label{questionk}
    Is there a function $\kappa: \mathbb{Z}_{>0}\times \mathbb{Z}_{>0}\to \mathbb{Z}_{>0}$ such that if $k\geq \kappa(n,d)$ then:
    
    Version I: degree $kd$ points on $X_{n,d}$ are not all supported on degree $k$ curves?
    
    Version II: degree $kd$ points on $X_{n,d}$ are dense when viewed as rational points on the
    
    symmetric power $\Sym^{kd} X_{n,d}$?
\end{question}

It would be interesting to obtain other examples of universal hypersurfaces with many rational multisections as in \Cref{cor:cubic-on-cubic}. The universal hypersurface of degree $d$ and dimension $n\gg d$ is geometrically unirational (as proved by Morin \cite{morin1940}; see also \cite{harris1998hypersurfaces}), which impacts their arithmetic.  In particular, one may expect that for $d$ and $k$ fixed and $n$ sufficiently large, $\mathcal{X}_{n,d}$ has many rational multisections of degree $kd$:

\begin{question}\label{question}
    Is there a function $\nu: \mathbb{Z}_{>0}\times \mathbb{Z}_{>0}\to \mathbb{Z}_{>0}$ such that if $n\geq \nu(d,k)$ then:

     Version I:  degree $kd$ points on $X_{n,d}$ are not all supported on degree $k$ curves?
    
    Version II: degree $kd$ points on $X_{n,d}$ are dense when viewed as rational points on the
    
    symmetric power $\Sym^{kd} X_{n,d}$?
\end{question}

In light of Theorem \ref{thm:main-introduction} and Lemma \ref{curves}, it is natural to ask whether an analogue holds true for points of degree $k(d_1 \cdots d_e)$ on the universal complete intersection of type $\underline{d}=(d_1, \cdots, d_e)$ in $\PP^{n+e}$. We propose the following conjecture.
\begin{conjecture}\label{conj:complete-intersection}
    Let $X_{n, \underline{d}}$ denote universal complete intersection of type $\underline{d}$ in $\PP^{n+e}$. There exists a constant $\delta(n,k)$ such that if $\sum d_i \geqslant \delta(n,k)$ every degree $k(d_1\cdots d_e)$ point of $X_{n, \underline{d}}$ arises from taking a complete intersection with a degree $k$ variety of dimension $e$.
\end{conjecture}

It is not clear how to generalize the method we used to prove Theorem \ref{thm:main-introduction} beyond the scope of the theorem to give new cases of Conjecture \ref{conj:complete-intersection}. One cause for concern is that the Cayley--Bacharach condition provided by the Mumford--Ro\u \i tman theorem (see \Cref{corCB}) is relative to the line bundle $K_X=\calO((\sum d_i)-(n+e)-1)$, whose positivity grows like $d_1+\dots+d_e$, which is small compared with $d_1\cdots d_e$. Consider the case of a complete intersection $X\subset \PP^{n+2}$ of two hypersurfaces of degrees $d_1, d_2$. Even if we knew that every degree $d_1d_2$ point satisfies the Cayley--Bacharach condition relative to $K_X=\calO(d_1+d_2-n-3)$, this does not appear to give much geometric information since $d_1d_2\gg d_1+d_2$.\\

Of course, one can formulate the questions and problems in this section for subvarieties of $X_{n,d}$ or $X_{n,\underline{d}}$ of higher dimension. We decided to emphasize the case of points for simplicity.\\

\bibliographystyle{amsalpha}
\bibliography{Bibliography.bib}

\end{document}